\documentclass[11pt]{amsart}
\usepackage{amsfonts,latexsym,amsthm,amssymb,amsmath,amscd,euscript,tikz, tikz-cd}
\usepackage[alphabetic, msc-links, bibtex-style, nobysame]{amsrefs}
\usepackage{stackengine}
\usepackage{framed}
\usepackage{xfrac}
\usepackage[makeroom]{cancel}
\usepackage{faktor}
\usepackage{braket}
\usepackage{pgf,tikz,pgfplots}
\pgfplotsset{compat=1.17}
\usepgfplotslibrary{fillbetween} 
\usepackage{mathrsfs}
\usetikzlibrary{arrows}
\definecolor{wrwrwr}{rgb}{0.3803921568627451,0.3803921568627451,0.3803921568627451}
\definecolor{rvwvcq}{rgb}{0.08235294117647059,0.396078431372549,0.7529411764705882}
\definecolor{mblue}{rgb}{0.2, 0.3, 0.8}
\definecolor{morange}{rgb}{1, 0.5, 0}
\definecolor{mgreen}{rgb}{0.1, 0.4, 0.2}
\definecolor{mred}{rgb}{0.5, 0, 0}

\usepackage{hyperref}
    \hypersetup{colorlinks=true,citecolor=blue,linkcolor = blue,urlcolor =black,linkbordercolor={1 0 0}}
\usepackage{float}

\numberwithin{equation}{section}

\usepackage{lmodern}
\usepackage{supertabular}
\usepackage{verbatim}
\usepackage{amssymb}
\usepackage{enumerate}
\usepackage{stmaryrd}
\usepackage{bbm}
\usepackage{mathtools}
\usepackage{microtype}
\usepackage{setspace}
\usepackage{tikz,tikz-cd}
\usetikzlibrary{matrix,calc,positioning,arrows,decorations.pathreplacing,patterns,knots}

\newcommand{\la}{\langle}
\newcommand{\rg}{\rangle}

\newtheorem{theorem}{{Theorem}}
\newtheorem*{theorem*}{Theorem}
\newtheorem{lemma}{Lemma}
\newtheorem{proposition}{Proposition}

\newtheorem{corollary}{Corollary}
\newtheorem*{corollary*}{Corollary}
\theoremstyle{definition}
\newtheorem{remark}{Remark}

\usepackage{lmodern,url,enumerate,mathtools,microtype}
\usepackage[hmargin = 1in,vmargin=1in]{geometry}
\usepackage{graphicx}
\usepackage{subcaption}

\newcommand{\ve}{\varepsilon}

\newcommand{\mr}[1]{{\rm #1}}
\newcommand{\cA}{\mathcal{A}}

\newcommand{\cE}{\mathcal{E}}\newcommand{\cF}{\mathcal{F}}

\newcommand{\cL}{\mathcal{L}}
\newcommand{\cM}{\mathcal{M}}

\newcommand{\cT}{\mathcal{T}}

\newcommand{\cW}{\mathcal{W}}

\newcommand{\bR}{\mathbb{R}}
\newcommand{\bS}{\mathbb{S}}

\newcommand{\fo}{\mathfrak{o}}

\newcommand{\nc}{\newcommand}

\nc{\on}{\operatorname}
\nc{\p}{\partial}
\nc{\ol}{\overline}
\nc{\ul}{\underline}
\nc{\pa}{\partial}

\nc{\pb}{\partial_b}
\nc{\pc}{\partial_c}
\nc{\pd}{\partial_d}
\nc{\pe}{\partial_e}
\nc{\pf}{\partial_f}
\nc{\pg}{\partial_g}
\nc{\ph}{\partial_h}
\nc{\pari}{\partial_i}
\nc{\pj}{\partial_j}
\nc{\pk}{\partial_k}
\nc{\pl}{\partial_l}
\nc{\pell}{\partial_\ell}
\nc{\parm}{\partial_m}
\nc{\pn}{\partial_n}
\nc{\po}{\partial_o}
\nc{\pp}{\partial_p}
\nc{\pq}{\partial_q}
\nc{\pr}{\partial_r}
\nc{\ps}{\partial_s}
\nc{\pt}{\partial_t}
\nc{\pu}{\partial_u}
\nc{\pv}{\partial_v}
\nc{\pw}{\partial_w}
\nc{\px}{\partial_x}
\nc{\py}{\partial_y}
\nc{\pz}{\partial_z}

\nc{\Sz}{Sz\'ekelyhidi\ }

\nc{\Spec}{\on{Spec}}

\nc{\sn}{\mr{sn}}
\nc{\cn}{\mr{cn}}
\nc{\dn}{\mr{dn}}

\newcommand{\bigslant}[2]{{\raisebox{.15em}{$#1$}/\raisebox{-.15em}{$#2$}}}

\numberwithin{equation}{section}
\makeatletter
\@addtoreset{equation}{section}
\makeatother

\title{Uniqueness of Cylindrical Tangent Cones $C_{p,q} \times \bR$} 
\date{\today}

\author[Benjy Firester]{Benjy Firester}
\address{
{\href{mailto:benjyfir@mit.edu}{benjyfir@mit.edu}}}
\author[Raphael Tsiamis]{Raphael Tsiamis}
\address{
{\href{mailto:r.tsiamis@columbia.edu}{r.tsiamis@columbia.edu}}}
\author[Yipeng Wang]{Yipeng Wang}
\address{
{\href{mailto:yw3631@columbia.edu}{yw3631@columbia.edu}}}

\begin{document}

\begin{abstract}
We show the uniqueness of the cylindrical tangent cone $C(\bS^2 \times \bS^4) \times \bR$ for area-minimizing hypersurfaces in $\bR^9$, completing the uniqueness of all tangent cones of the form $C_{p,q} \times \bR$ proved by Simon for dimensions at least 10 and Sz\'ekelyhidi for the Simons cone.
\end{abstract}

\maketitle
\vspace{-0.7cm}

\section{Introduction}
A tangent cone $C$ of an area-minimizing hypersurface $M$ is a blow-up limit of rescalings of the surface at a point. 
Generally, the minimal cone $C$ may depend on the blow-up sequence; when it is unique, we obtain powerful information about the regularity of the surface. 
The uniqueness of tangent cones with isolated singularities was shown in the celebrated results of Allard-Almgren~\cite{AllardAlmgren81} in the integrable case and Simon \cite{SimonAnnals81} for non-integrable cones.
For cones with non-isolated singularities, for example, of the form $C \times \bR^k$, the uniqueness question is largely open.

In \cite{uniqueness-cylindrical}, Simon proved the uniqueness of multiplicity one tangent cones of the form $C \times \bR$ for a large class of minimal cones $C$, which includes all quadratic cones $C_{p,q} = C(\bS^p \times \bS^q)$ with $p+q \geq 7$, using an integrability condition for their Jacobi fields.
The $7$-dimensional quadratic cones $C(\bS^3 \times \bS^3)$ and $C(\bS^2 \times \bS^4)$ require a more delicate treatment since they carry a non-integrable Jacobi field; in the case of the Simons cone, the corresponding uniqueness result is proved in \cite{uniqueness-gabor}*{Theorem 1.1}.
Our Theorem completes the uniqueness result for all tangent cones of the form $C_{p,q} \times \bR$.
\begin{theorem}\label{thm}
Let $M$ be an area-minimizing hypersurface in a neighborhood of $0 \in \bR^9$ that admits $C_{2,4} \times \bR$ as a multiplicity one tangent cone at $0$.
Then, $C_{2,4} \times \bR$ is the unique tangent cone at $0$.
\end{theorem}
The proof relies on desingularizing the link $\Sigma_0$ of $C_{2,4}\times \bR$ by gluing rescaled copies of the Hardt-Simon foliation at the singular points of $\Sigma_0$ to the graph of the non-integrable Jacobi field, artificially restoring a notion of integrability.
The desingularized surfaces will be almost minimal, and we obtain an asymptotic expansion for the Hardt-Simon foliation $C_{2,4}$ to produce a definite term in the mean curvature of the desingularizations. 
Unlike the Simons cone, $C_{2,4}$ has no additional symmetry to exploit; notably, there are two distinct Hardt-Simon foliations. 
The main new ingredient in our proof, Proposition \ref{prop:main-expansion}, is a precise understanding of the sub-leading term in the asymptotic expansion of the Hardt-Simon leaves as graphs over $C_{2,4}$, with which we can apply the \L{}ojasiewicz-type inequality and non-concentration result for cones developed in \cite{uniqueness-gabor}.
The proof of Proposition \ref{prop:main-expansion} relies on a particular reparametrization of the Hardt-Simon leaves and constructing precise barriers to prevent them from exhibiting non-analytic behavior.
Our technique can be applied algorithmically to compute higher-order asymptotic behavior of solutions to more general differential equations.

\section{Main Result}

We denote the points of $\bR^8$ by $\xi = (x,y) \in \bR^5 \times \bR^3$ and set $u = |x|, v = |y|$, so $u^2 + v^2 = r^2$.
We work with the area-minimizing cone in $\bR^8$ defined by $\mathbf{C}_0 = \{ (x,y) \in \bR^8 : u^2 = 2v^2 \}$, which is the quadratic cone $C(\bS^4 \times \bS^2)$.
Let $z$ be the coordinate along the added $\bR$-direction on the cone $\mathbf{C}_0 \times \bR$.

The homogeneous Jacobi fields on $\mathbf{C}_0$ can be well understood by a polar decomposition of the Jacobi operator.
This will play a critical step to invert the Jacobi operator, in particular, we isolate a range of degrees where there are no Jacobi fields meaning that in a proper weighted space $C^{k,\alpha}_\tau$, the kernel of the Jacobi operator $\cL$ is trivial.
For $n = 8$, there are three types of Jacobi fields whose degree lies in the range $(-2, 1]$ consisting of degree 0 translations, degree 1 rotations (generated by $O(p+1)\times O(q+1)$), and the specified Jacobi field coming from the Hardt-Simon smoothing, which has degree $-2$. 
Notably, in dimension 8, there are no Jacobi fields with degree in the range $(-3,-2)$. 
These can be computed using separation of variables for the Jacobi operator, upon obtaining the explicit spectrum of the Laplacian on $\bS^p \big(\sqrt{\frac{p}{n-2}}\big) \times \bS^q \big(\sqrt{\frac{q}{n-2}}\big)$; see \cite{SimonSolomon} for more details. 

The minimal cone $\mathbf{C}_0 \times \bR$ has a degree one homogeneous Jacobi field given by $\phi = z^3r^{-2} - z$, which also descends to a Jacobi field on the singular link $\Sigma_0 \subset \bS^8$, blowing up at the rate of $r^{-2}$ near the singular points.
This blow-up rate near the singular points will match the asymptotic rate of the Hardt-Simon foliation, which enables the construction of an almost minimal desingularization of $\Sigma_0$ via a gluing and perturbative technique.
On $\mathbf{C}_0 \times \bR$, $\phi$ and the rotations $\bigslant{\fo(9)}{\fo(5) \oplus \fo(3)}$ are the only degree one Jacobi fields. 
The results of \cite{uniqueness-gabor}*{Lemmas 2.2 -- 2.4} can be shown to apply to the cone $\mathbf{C}_0\times \bR$, with the same proofs as for the Simons cone $C_{3,3}\times \bR$.
 
Let $U_{+}=\{u>\sqrt{2}v\}$ and $U_-=\{u<\sqrt{2}v\}$ denote the two connected components of $\bR^8\setminus \mathbf{C}_0$, and we choose a unit normal vector field $\nu_{\mathbf{C}_0}(\xi)$ that points into $U_+$ for all $x\in \mathbf{C}_0\setminus \{0\}$, and it follows from a simple computation that $|\xi|\,\nu_{\mathbf{C}_0}(\xi)=\frac{\sqrt{2}}{2} (x,-2y)$. By the work of Hardt-Simon \cite{Hardt-Simon}, there exist smooth area-minimizing hypersurfaces $S_{\pm}\subset U_{\pm}$, and some constant $R_0$, such that outside $B_{R_0}\subset \bR^8$, $S_{\pm}$ can be represented as the normal graphs of some smooth functions $r f_{\pm}(r)$, where $r = |\xi|$, in the direction of the unit normal $\nu_{\mathbf{C}_0}(\xi)$.
We therefore express
\[
S_{\pm}=\left\{ (x,y) \pm f_\pm(r) (x, - 2y) \right\}.
\]
Moreover, \cite{Hardt-Simon}*{Theorem 3.2} shows that $|f_{\pm}- r^{-3}|\le O(r^{-3-\ve_0})$ as $r\to\infty$ for some $\ve_0>0$, and $U_{\pm}$ is foliated by dilations of $S_{\pm}$.
Each $f_{\pm}$ solves an equation of the form $r^2f_{\pm}''+8rf_{\pm}'+12f_{\pm}+\cE=0$, where $\cE$ contains at least quadratic terms involving $f_{\pm},rf_{\pm}'$ and $r^2f_{\pm}''$.
Since $|f_{\pm}-r^{-3}|=O(r^{-3-\ve_0})$ and the only solution of $r^2\psi''+8r\psi'+12\psi=0$ satisfying the decay $|\psi|\le O(r^{-3-\ve_0})$ is $\psi=b\,r^{-4}$ for some $b\in \bR$, we obtain some $b_{\pm} \in \bR$ so that
    \begin{equation}{\label{eqn:b-expansion}}
        |f_{\pm}-r^{-3}-b_{\pm}r^{-4}|= O(r^{-6}) \qquad \text{as } \; r \to \infty.
    \end{equation}
In the following, we will parametrize the hypersurfaces $S_{\pm}$ modulo the $O(5)\times O(3)$ symmetry in terms of the coordinates $(u,v) = e^{\varphi(t)} (\cos t, \sin t)$.
The parameter $t \in [0,t_0)$ corresponds to $S_+$ and $t \in (t_0, \frac{\pi}{2} ]$ to $S_-$, for $t_0$ satisfying $\cos^2(t_0) = 2 \sin^2(t_0)$.
Davini \cite{davini}*{\S 1} proved that $\varphi(t)$ satisfies
\begin{equation}{\label{eqn:Davini-eqn}}
\frac{d^2\varphi }{dt^2}=\left(1+\left(\frac{d\varphi}{dt}\right)^2\right)\,\left[7+\left(\frac{2-6\cos(2t)}{\sin(2t)}\right)\frac{d\varphi}{dt}\right],
\end{equation}
with $\varphi'(0)=0$. 
We will denote $w(t)=\varphi'(t)$ and consider the reparametrization $s(t)=\sin(2t)\slash \sin (2(t_0-t))$. Notice that $s(0)=0$ and $\lim_{t\nearrow t_0}s(t)\to \infty$ parametrizes the leaf $S_+$.
\begin{lemma}\label{lemma:main-lemma}
    Expressed as a function of $s \in [0,\infty)$, the function $w(s)$ solves the equation
    \begin{equation}{\label{eqn:main-ODE}}
       s\,\frac{dw}{ds}=(1+w^2)\left(\frac{7\sqrt{2}s+(\sqrt{9s^2+6s+9}-3s-9)w}{3s^2+2s+3}\right)
    \end{equation}
    for $s\in [0,\infty)$ with $w(0)=0$.
    As $s \to \infty$, we have
    \begin{equation}{\label{eqn:w-asymptotic}}
       w=\frac{\sqrt{2}}{2}s-\frac{(\tfrac{3\sqrt{2}}{2})^{\frac{2}{3}}\,b_+}{9} s^{\frac{2}{3}}+\frac{5\,(\tfrac{3\sqrt{2}}{2})^{\frac{1}{3}}\,b_+^2}{27}s^{\frac{1}{3}}+O(1)
    \end{equation}
    where $b_+$ is the coefficient of the expansion \eqref{eqn:b-expansion}. 
\end{lemma}
\begin{proof}
Since $\cos(2t_0)=\frac{1}{3}$, we have $s^{-1}=\tfrac{2\sqrt{2}}{3}\cot(2t)-\tfrac{1}{3}$ and $\cot(2t)= \frac{3s^{-1}+1}{2\sqrt{2}}$. It then follows that
    \[
    \sin(2t)=\frac{1}{\sqrt{1+\cot^2(2t)}}=\frac{2\sqrt{2}s}{\sqrt{9s^2+6s+9}}\quad \text{ and }\quad
        \cos(2t)=\frac{\cot^2(2t)}{\sqrt{1+\cot^2(2t)}}=\frac{s+3}{\sqrt{9s^2+6s+9}}.
    \]
    In particular, we can express
    \[
    \frac{2-6\cos(2t)}{\sin(2t)}=\frac{\sqrt{9s^2+6s+9}-3s-9}{\sqrt{2}s}.
    \]
    Moreover, differentiating the identity  \(
    s^{-1}=\tfrac{2\sqrt{2}}{3}\cot(2t)-\tfrac{1}{3}
    \)
    with respect to $t$ yields
    \begin{equation}{\label{eqn:s-t-derivative}}
    s^{-2}\frac{ds}{dt}=\frac{4\sqrt{2}}{3\sin^2(2t)}=\frac{3s^2+2s+3}{\sqrt{2}s^2}.
    \end{equation}
    Using equation \eqref{eqn:Davini-eqn}, we conclude that
    \begin{align*}
        \frac{dw}{ds}&=\frac{dw}{dt}\cdot \frac{dt}{ds} = (1+w^2)\left(7+\frac{\sqrt{9s^2+6s+9}-3s-9}{\sqrt{2}s}w\right)\cdot\frac{\sqrt{2}}{3s^2+2s+3},
    \end{align*}
    which proves \eqref{eqn:main-ODE}.
    To obtain the expansion \eqref{eqn:b-expansion}, we represent the points $\xi + |\xi| f_+(|\xi|) \nu_{\mathbf{C}_0}(\xi) \in S \setminus B_{R_0}$ by $(u,v)=(e^{\varphi(t)}\,\cos (t),e^{\varphi(t)}\,\sin (t))$.
    Since $|\xi|\,\nu_{\mathbf{C}_0}(\xi)=\tfrac{\sqrt{2}}{2}(x,-2y)$, we can write
    \[
    \xi+|\xi|\,f_+(|\xi|)\nu_{\mathbf{C}_0}(\xi)=(x,y)+ \frac{\sqrt{2}}{2}f_+(r) (x,-2y),
    \]
    for some $\xi = (x,y) \in \mathbf{C}_0$. 
    Using $u^2 + v^2 = r^2$ and $u^2 = 2v^2$ on $\mathbf{C}_0$, we obtain
    \begin{align*}
    e^{2\varphi(t)}&=\left(1+\frac{f_+}{\sqrt{2}} \right)^2u^2+\left(1-\frac{2f_+}{\sqrt{2}} \right)^2v^2 = (1 + f_+(r)^2) \, r^2.
    \end{align*}
    Using the asymptotics $|f_+-r^{-3}|=O(r^{-4})$, we arrive at
    \begin{equation}\label{eqn:phi-asymptotics}
        |\varphi(t) - \log r| = O(r^{-6}) \quad \text{as } \; r \to \infty.
    \end{equation}
    We now express 
    $$\tan^2 t=\frac{\left(1-\frac{2f_+}{\sqrt{2}} \right)^2v^2}{\left(1+\frac{f_+}{\sqrt{2}} \right)^2u^2}=\frac{1}{2}\frac{\left(1-\frac{2f_+}{\sqrt{2}} \right)^2}{\left(1+\frac{f_+}{\sqrt{2}} \right)^2}$$
    in order to obtain
    \[
        s=\frac{3\sin(2t)}{2\sqrt{2}\,\cos(2t)-\sin(2t)} =\frac{6\tan t}{2\sqrt{2}(1-\tan^2 t)-2\tan t}=\frac{\sqrt{2}}{3}f_+^{-1}+O(1).
    \]
    Using the expansion $|f_+-r^{-3}-b_+\,r^{-4}|=O(r^{-6})$ from \eqref{eqn:b-expansion}, we have 
    $$s= \frac{\sqrt{2}}{3}(r^{3}-b_+r^{2}+b_+^2 r)+O({1}).$$ 
    We can therefore compute
    \[
    s^{\frac{1}{3}}= \bigl(\tfrac{\sqrt{2}}{3} \bigr)^{\frac{1}{3}} \bigg(r-\frac{b_+}{3}+\frac{2b_+^2}{9r}\bigg)+O(r^{-2}) \quad \text{ and }\quad s^{-\frac{1}{3}}= \bigl(\tfrac{\sqrt{2}}{3} \bigr)^{-\frac{1}{3}} \left(r^{-1}+\frac{b_+}{3r^2}\right)+O(r^{-3})
    \]
    as $r\to\infty$, so as $s\to\infty$,
    \[
    r= (\tfrac{\sqrt{2}}{3})^{-\frac{1}{3}} s^{\frac{1}{3}}+ \frac{b_+}{3}- \frac{2 (\frac{\sqrt{2}}{3})^{\frac{1}{3}} b_+^2}{9}s^{-\frac{1}{3}}+O(s^{-\frac{2}{3}}).
    \]
    It then follows from \eqref{eqn:phi-asymptotics} that
    \[
    \varphi=\frac{1}{3}\log s-\frac{\frac{1}{2} \log 2 - \log 3}{3} +\frac{(\frac{\sqrt{2}}{3})^{\frac{1}{3}} b_+}{3}s^{-\frac{1}{3}}-\frac{5 (\frac{\sqrt{2}}{3})^{\frac{2}{3}} \, b_+^2}{18} s^{-\frac{2}{3}}+O(s^{-1}).
    \]
    Differentiating the above with respect to $s$ shows that
    \begin{align*}
    \frac{d\varphi}{ds}&=\frac{1}{3s}- \frac{ (\frac{\sqrt{2}}{3})^{\frac{1}{3}} b_+}{9}s^{-\frac{4}{3}}+ \frac{5 (\frac{\sqrt{2}}{3})^{\frac{1}{3}} \,b_+^2}{27}s^{-\frac{5}{3}}+O(s^{-2}).
    \end{align*}
    Multiplying the above identities by \eqref{eqn:s-t-derivative} therefore yields the claimed expansion \eqref{eqn:w-asymptotic}.
\end{proof}
The above argument also applies to the analysis of the leaf $S_-$, parametrized by $(u,v) = e^{\varphi(t)}(\cos t, \sin t)$ for $t \in (t_0, \frac{\pi}{2}]$.
Introducing the variable $\hat{s} := - s = \frac{\sin(2t)}{\sin(2(t-t_0))}$, we have $\hat{s}(\frac{\pi}{2}) = 0$ and $\lim_{t \searrow s_0} \hat{s}(t) \to \infty$.
As in Lemma \ref{lemma:main-lemma}, we compute
\begin{align*}
\sin(2t) &= \frac{ 2 \sqrt{2} \, \hat{s}}{ \sqrt{ 9 \hat{s}^2 - 6 \hat{s} + 9}} \qquad \text{and} \qquad \cos(2t) = \frac{ \hat{s} - 3 }{ \sqrt{ 9 \hat{s}^2 - 6 \hat{s} + 9}}, \\
\frac{d \hat{s}}{dt} &= - \frac{3 \hat{s}^2 - 2 \hat{s} + 3}{\sqrt{2}}.
\end{align*}
We express $\varphi'(t)$ for $t \in (t_0, \frac{\pi}{2}]$ as a function of $\hat{s} \in (0, \infty)$, denoted by $\hat{w}(\hat{s})$.
The same computations as in Lemma \ref{lemma:main-lemma} yield $\varphi(t) = \frac{1}{2} \log (r^2 (1+f_-^2)) = \log r + O(r^{-6})$, and we obtain expressions for $\hat{w}$ analogous to \eqref{eqn:main-ODE} and \eqref{eqn:w-asymptotic}.
\begin{corollary}
    In the above notation, the function $\hat{w}(\hat{s})$ solves the equation
\begin{equation}\label{eqn:w-hat-for-s-hat}
   \hat{s} \frac{d \hat{w}}{d \hat{s}} = - (1 + \hat{w}^2) \left( \frac{7 \sqrt{2} \, \hat{s} + ( \sqrt{9 \hat{s}^2 - 6 \hat{s} + 9} - 3 \hat{s} + 9) \hat{w}}{3 \hat{s}^2 - 2 \hat{s} + 3} \right) 
\end{equation}
with initial condition $\hat{w}(0) = 0$.
Moreover, $\hat{w}(\hat{s})$ admits the expansion
\begin{equation}\label{eqn:w-hat-s-expansion}
    \hat{w}(\hat{s}) = - \frac{\sqrt{2}}{2} \hat{s} +\frac{(\tfrac{3\sqrt{2}}{2})^{\frac{2}{3}}\,b_-}{9}\hat{s}^{\frac{2}{3}}-\frac{5\,(\tfrac{3\sqrt{2}}{2})^{\frac{1}{3}}\,b_-^2}{27}\hat{s}^{\frac{1}{3}}+O(1).
\end{equation}
\end{corollary}
Combining the above results, we refine the asymptotic expansion of the Hardt-Simon leaves $S_{\pm}$.
\begin{proposition}{\label{prop:main-expansion}}
    There exist constants $b_+ > 0$ and $b_- < 0$, such that
    \[
    |f_+(r) - r^{-3} - b_+ r^{-4}| = O(r^{-6}) \qquad \text{and} \qquad |f_-(r) - r^{-3} - b_- r^{-4}| = O(r^{-6})
    \] 
    as $r\to\infty$.
    These constants satisfy 
    \begin{equation}\label{eqn:b+-Estimates}
     0 <\frac{(\tfrac{3\sqrt{2}}{2})^{\frac{2}{3}}}{9} b_+\leq 0.1  \qquad \text{and} \qquad  \frac{(\tfrac{3\sqrt{2}}{2})^{\frac{2}{3}}}{9} b_- \leq -0.11.
    \end{equation}
    Consequently, the sum $(b_+ + b_-) \leq - 9 \cdot 10^{-2} \cdot(\tfrac{3\sqrt{2}}{2})^{-\frac{2}{3}} < 0$ is non-zero.
\end{proposition}
The above estimates on $b_{\pm}$ are crude, aiming to use the simplest possible barriers in the proof below to show that $(b_+ + b_-) < 0$.
Using more accurately chosen barriers, one can show that $b_+ \approx 0.22$ and $b_- \approx - 0.93$. 
\begin{proof}
In view of the expansion \eqref{eqn:w-asymptotic}, we first prove that $b_+>0$ by showing that the function $w(s) - \frac{\sqrt{2}}{2}s$ is unbounded below.
    We define
    \[
    G(s,w) := s \frac{dw}{ds} - \frac{7 \sqrt{2}}{3} s + 2w,
    \]
    which from equation~\eqref{eqn:main-ODE} is an analytic function containing at least quadratic terms in evolving $s$ and $w$.
    Since $G(s, \tfrac{7 \sqrt{2}}{9} s) = 0$, standard ODE theory shows $|w(s) - \frac{7 \sqrt{2}}{9}s| = O(s^2)$ for all sufficiently small $s$.
    We now claim that the function $h(s) := \frac{\sqrt{2}}{2} s + \frac{2s}{2+5s}$ is a supersolution of \eqref{eqn:main-ODE}, meaning that the function
    \begin{equation}\label{eqn:F-g-supersolution}
        F(h(s)) := s h'(s) - (1+h^2)\left(\frac{7\sqrt{2}s+(\sqrt{9s^2+6s+9}-3s-9) \,h}{3s^2+2s+3}\right)
    \end{equation}
    satisfies $F(h(s)) \geq 0$.
    We write
    \[
    F(h(s)) = s \frac{R_1(s) - R_2(s) \sqrt{s^2 + \frac{2}{3} s + 1}}{(2 + 5 s)^3 \,(3 + 2 s + 3 s^2)}
    \]
    for polynomials $R_1, R_2$ of degrees $6$ and $5$.
    We then compute that $\tilde{F}(h(s)) := R_1(s)^2 - R_2(s)^2 (s^2 + \frac{2}{3} s +1)$ is the degree-ten polynomial
    \begin{align*}
        \tilde{F}(s)&= (8480 - 2112\sqrt{2}) + 
 32 (2467 - 707\sqrt{2}) s + (286776 - 82496\sqrt{2}) s^2 + 
 112 (4061 - 730\sqrt{2}) s^3 \\
 & \quad + 2 (71761 + 97590\sqrt{2}) s^4 +(517938\sqrt{2}-382130 ) s^5 + 
 \tfrac{5}{2} (122924\sqrt{2}-99677 ) s^6 \\ 
 & \quad +(86075 - 42225\sqrt{2}) s^7 - 
 \tfrac{125}{2} (395 - 162\sqrt{2}) s^8 - \tfrac{625}{2} (44\sqrt{2}-41) s^9 + 
 6250 (5 - 3\sqrt{2}) s^{10}.
    \end{align*}
Since $\tilde{F}(0)>0$ and the coefficients of $1, s, \dots, s^6$ are positive, we can bound 
\begin{align*}
    \tilde{F}(s)
 &\geq 25000\, s^7 - \tfrac{125}{2} \cdot 175\, s^8 - \tfrac{625}{2}\cdot 30\,s^9  + \tfrac{3}{4}\cdot6250\,  s^{10}\\
 &= \tfrac{3125}{2} s^7 (16 - 7 s - 6 s^2 + 3 s^3),
\end{align*}
where the cubic $Q(s) := 16 - 7s - 6s^2 + 3s^3$ is positive for all $s > 0$, with $Q(s) > Q ( \frac{2 + \sqrt{11}}{3}) > 1$.
This shows that $h(s)$ is a supersolution of \eqref{eqn:main-ODE}.
    Since $|h(s) - \frac{2+\sqrt{2}}{2}s| = O(s^2)$, we have $w(s) < h(s)$ for $s$ sufficiently small, where $w < h$ for all $s \in (0 ,\infty)$.
    In particular, the function $\ve(s) := w(s) - \frac{\sqrt{2}}{2}$ satisfies $\sup_{(0, \infty)} \ve(s) \leq \sup_{(0,\infty)} \frac{2s}{2 + 5s} = \frac{2}{5}$.
    Suppose for contradiction that $\ve(s) > - C_1$ is bounded below, for some $C_1>0$, so that $- C_1 < \ve(s) \leq  \frac{2}{5}$ for all $s \in (0,\infty)$.
    For all $s$ large, we have
    \[
    0 < \left( \sqrt{9s^2 + 6 s + 9} - 3s - 9 \right) - \left( - 8 + \tfrac{4}{3} s^{-1} \right) < \tfrac{1}{2} s^{-2}
    \]
    and $1 + w^2 = s^2\left[\frac{1}{2}+ \frac{\sqrt{2}\ve}{s} + \frac{1 + \ve^2}{s^2}\right]$.
    Since $w > 0$ for $s$ large, we apply this in \eqref{eqn:main-ODE} combined with the Taylor expansion bound $\frac{1}{3s^2 + 2s + 3} \leq \frac{1}{3s^2}\left(1 - \frac{2}{3s} \right)$ and the bound $\ve(s) \leq \frac{2}{5}$ to show
    \begin{align*}
        s \frac{dw}{ds} 
        &\leq \frac{1}{3s^2} \left( 1 - \frac{2}{3s} \right) \cdot s^2 \left( \frac{1}{2} + \frac{\sqrt{2} \ve}{s} + \frac{1+\ve^2}{s^2} \right) \cdot \left(3\sqrt{2} s + \frac{2\sqrt{2}}{3} - 8\ve + \frac{4\ve}{3s}\right) \\
        &\leq \frac{\sqrt{2}}{2} s+ \left(\frac{2\ve}{3} -  \frac{2\sqrt{2}}{9} \right) + O(s^{-1}) \\
        &\leq \frac{\sqrt{2}}{2} s - \frac{1}{25}.
    \end{align*}
Therefore, for $s$ sufficiently large, we arrive at the differential inequality $w' - \frac{\sqrt{2}}{2} < - \frac{1}{25 s}$, proving that $w - \frac{\sqrt{2}}{2}s \leq C_2 - \frac{1}{25} \log s$ is unbounded.
    This contradicts our original assumption, so $b_+>0$.

    To prove the upper bound on $b_+$, we produce a subsolution $g(s)$ to the ODE \eqref{eqn:main-ODE} for $w$, given by
    \begin{equation}\label{eqn:g-subsolution-for-b+}
    g(s) := \begin{cases}
    \tfrac{104}{95} s -\tfrac{1}{5} s^{\frac{3}{2}} & s \in [0,1), \\
    - \frac{1}{840} s^4 + \frac{1}{64} s^3 - \frac{9}{110} s^2 + \frac{83}{91} s + \frac{1}{20}, & s \in [1,\tfrac{81}{20}), \\
    \tfrac{\sqrt2}{2} s + \tfrac{1}{4} - \tfrac{3}{250} (s - \tfrac{81}{20}) - \tfrac{2}{10^5} \, (s - \tfrac{81}{20})^2, & s \in [\tfrac{81}{20}, 37), \\
    \tfrac{\sqrt{2}}{2}s - \tfrac{1}{10} s^{\frac{2}{3}} + \tfrac{\sqrt{2}}{5} s^{\frac{1}{3}}, & s \in [37, \infty).
        \end{cases}
    \end{equation}
    We verify that $g$ is a subsolution of \eqref{eqn:main-ODE} in each of the four regions, meaning that the expression $F(g(s))$ defined as in \eqref{eqn:F-g-supersolution} satisfies $F(g(s)) \leq 0$.
    We first prove that $F(g(s)) \leq 0$ on $[37, \infty)$.
    We compute $F(g(s))$ as done previously, for $F(h(s))$ following \eqref{eqn:F-g-supersolution}, to write
    \begin{equation}\label{eqn:q1-q2-s-13}
    F(g(s)) = \frac{Q_1(s^{\frac{1}{3}}) - Q_2(s^{\frac{1}{3}}) \sqrt{s^2 + \frac{2}{3} s + 1}}{3s^2 + 2s + 3}
    \end{equation}
    for polynomials $Q_1, Q_2$ in the variable $\tau = s^{\frac{1}{3}}$.
    To prove that $F(g(s)) \leq 0$, we form the expression
    \[
    \tilde{F}(\tau) := Q_1(\tau)^2 - Q_2(\tau)^2 (\tau^6 + \tfrac{2}{3} \tau^3 + 1) = \tau^2 Q_3(\tau)
    \]
    as done for $\tilde{F}(h(s))$ above.
    We can compute $Q_3(\tau)$, which is the degree-$17$ polynomial given by
    \begin{align*}
        Q_3(\tau) &= \tfrac{182}{25} - \tfrac{101 \sqrt{2}}{25} \tau - \tfrac{5902}{625} \tau^2 + (\tfrac{404}{75}+\tfrac{2543\sqrt{2}}{1250} )\tau^3 - (\tfrac{233\sqrt{2}}{75}-\tfrac{24677}{31250})\tau^4 + (\tfrac{10181}{625}-\tfrac{132731\sqrt{2}}{62500})\tau^5 \\
        & \; \; - (\tfrac{439979}{112500} +\tfrac{17969\sqrt{2}}{3750} ) \tau^6 - (\tfrac{351548}{46875}- \tfrac{383491\sqrt{2}}{225000})\tau^7 - (\tfrac{196903}{112500} + \tfrac{739537\sqrt{2}}{187500} )\tau^8 + (\tfrac{80003}{7500}-\tfrac{3137\sqrt{2}}{187500}) \tau^9 \\
        & \; \; - (\tfrac{125969\sqrt{2}}{25000}-\tfrac{251359}{125000})\tau^{10} + (\tfrac{71021}{37500}+\tfrac{3379\sqrt{2}}{10000} )\tau^{11} - (\tfrac{867}{2500}+\tfrac{301283\sqrt{2}}{187500}) \tau^{12} - (\tfrac{18443}{15625} + \tfrac{17\sqrt{2}}{1000})\tau^{13} \\
        & \; \; - (\tfrac{67671\sqrt{2}}{50000}-\tfrac{29}{20}) \tau^{14} + (\tfrac{5721}{5000}-\tfrac{\sqrt{2}}{2})\tau^{15} + (2-\tfrac{477\sqrt{2}}{1000})\tau^{16} - \tfrac{9}{20} \tau^{17}.
    \end{align*}
     We can trivially bound $Q_3$ from above by removing most terms with negative coefficients and rounding-up the remaining coefficients.
     Also, for all $\tau \geq (37)^{\frac{1}{3}}$, we can bound
    \[
    8+9\tau^3 +14 \tau^5+11\tau^9+\tfrac{5}{2}\tau^{11}\le 10 \tau^{11} \qquad \text{and} \qquad -\tfrac{2}{5}\tau^{14}\le -\tfrac{74}{5}\tau^{11}.
    \]
    Combining these steps, we may estimate $Q_3(\tau)$ by
    \begin{align*}
    Q_3(\tau) &< 8+9\tau^3 +14 \tau^5+11\tau^9+\tfrac{5}{2}\tau^{11}-\tfrac{2}{5}\tau^{14}+\tfrac{9}{20}\tau^{15}+\tfrac{27}{20} \tau^{16}-\tfrac{9}{20} \tau^{17} \\ 
    &< \tfrac{9}{20}\tau^{15}+\tfrac{27}{20} \tau^{16}-\tfrac{9}{20} \tau^{17} \\
    &= \tfrac{9}{20} \tau^{15} (1 + \tau(3-\tau)).
    \end{align*}
    For all $\tau \geq (37)^{\frac{1}{3}}$ we have $1+\tau(3-\tau) \leq 1 - (37)^{\frac{1}{3}} ((37)^{\frac{1}{3}}-3) < - \frac{1}{10}$, so $Q_3(\tau) < 0$ shows that $g(s)$ is a subsolution on $[37, \infty)$.
    
    On $[0,1)$, we again obtain an expression of the form \eqref{eqn:q1-q2-s-13}, where $Q_1$ and $Q_2$ are now polynomials in $\tau = s^{\frac{1}{2}}$.
    A computation similar to the above shows that
    \[
    \tilde{F}(\tau) := Q_1(\tau)^2 - Q_2(\tau)^2 (\tau^4 + \tfrac{2}{3} \tau^2 + 1) = \tau^4 Q_3(\tau),
    \]
    where $Q_3(\tau)$ is degree-$16$ polynomial with $\sup_{[0,1]} Q_3 < - 10^{-2}$.
    Therefore, $\tilde{F}(\tau) \leq 0$ for all $\tau \in [0,1]$ makes $g(s)$ a subsolution on $[0,1]$.
    
    On $[1, \frac{81}{20})$, the quartic polynomial defining $g(s)$ is determined by seeking a polynomial approximate solution of the ODE \eqref{eqn:main-ODE} on this interval with initial condition $\tilde{w}(1) = \lim_{s \nearrow 1} g(1)$, and truncating the coefficients appropriately to guarantee $\lim_{s \nearrow \frac{81}{20}} g(s) > \frac{1}{\sqrt{2}} \frac{81}{20} + \frac{1}{4}$.
    We may compute $F(g(s))$ as done above, in \eqref{eqn:q1-q2-s-13}, to write
    \begin{equation}\label{eqn:q1-q2-s}
    F(g(s)) = \frac{Q_1(s) - Q_2(s) \sqrt{s^2 + \frac{2}{3} s + 1}}{3s^2 + 2s + 3}
    \end{equation}
    for polynomials $Q_1, Q_2$ of degrees $13$ and $12$, respectively.
    Forming the expression $\tilde{F}(g(s)) := Q_1(s)^2 - Q_2(s)^2 (s^2 + \frac{2}{3} s+ 1)$ as above produces a degree-$25$ polynomial (the degree $26$ terms cancel) whose non-positivity on $[1, \frac{81}{20})$ is obtained by a similar computation to the one used for $\tilde{F}(h(s))$ previously.
    In fact, we have the conservative estimate $\tilde{F}(g(s)) < - \frac{1}{10}$, which gives $F(g(s)) < - 10^{-4} < 0$ everywhere in this region.
    The same argument is valid, and in fact simpler, on $[\frac{81}{20}, 37)$, where working with a degree-$13$ polynomial yields $F(g(s)) < - 10^{-3} < 0$ in this interval.
     
    The function $g(s)$ of \eqref{eqn:g-subsolution-for-b+} satisfies the subsolution condition in each region and is piecewise continuous with decreasing jump discontinuities at $s \in \{ 1,  \tfrac{81}{20}, 37 \}$.
    At those points, we verify that
    \begin{align*}
        \lim_{s \nearrow 1} g(s) &= \tfrac{17}{19} > \lim_{s \searrow 1} g(s), \\
        \lim_{s \nearrow \tfrac{81}{20}} g(s) &> \tfrac{\sqrt{2}}{2} \tfrac{81}{20} + \tfrac{1}{4} + 10^{-4} > \lim_{s \searrow \tfrac{81}{20}} g(s), \\
        \lim_{s \nearrow 37} g(s) &> 26 - 45 \cdot 10^{-4} > \lim_{s \searrow 37} g(s),
    \end{align*}
    so $g(s)$ is a global subsolution.
    Moreover, $g'(0) = \tfrac{104}{95}  < \tfrac{7\sqrt{2}}{9} = w'(0)$ shows that $g(s) < w(s)$ for small $s>0$, hence $g(s) < w(s)$ for all $s \in (0,\infty)$.
    Using the expansion \eqref{eqn:w-asymptotic} for $w$ and the form of $g$ as $s \to \infty$, we obtain the bound on $b_+$ in the expression~\eqref{eqn:b+-Estimates}. 

    The corresponding bound \eqref{eqn:b+-Estimates} for $b_-$ and the surface $S_- \subset U_-$ of $\mathbf{C}_0$ follow by a similar analysis of equations \eqref{eqn:w-hat-for-s-hat} and \eqref{eqn:w-hat-s-expansion} for $\hat{w}$.
    An analogous computation to the above shows that the function
    \begin{equation}\label{eqn:g-supersolution-for-b-}
    \hat{g}(\hat{s}) := \begin{cases}
    - \tfrac{\sqrt{2}}{2} \hat{s} - \tfrac{2}{5} \hat{s}^{\frac{2}{3}} + \tfrac{3}{10} \hat{s}^{\frac{1}{3}}, & \hat{s} \in [0,3), \\
    - \tfrac{\sqrt{2}}{2} \hat{s} - 0.11 \, \hat{s}^{\frac{2}{3}} - 5 \sqrt{2} \, (0.11)^2 \, \hat{s}^{\frac{1}{3}}, & \hat{s} \in [3, \infty),
        \end{cases}
    \end{equation}
    is a supersolution to the ODE \eqref{eqn:w-hat-for-s-hat} for all $\hat{s} \in [0, \infty)$.
    This is a piecewise continuous function with an increasing jump discontinuity at $\hat{s} = 3$, where
    \[
    \lim_{\hat{s} \nearrow 3} \hat{g}(s) < - \tfrac{5}{2} < \lim_{\hat{s} \searrow 3} \hat{g}(s),
    \]
    so it is a global supersolution.
    Arguing as above, we obtain $|\hat{w}(\hat{s}) + \frac{7 \sqrt{2}}{15} \hat{s}| = O(\hat{s}^2)$, hence $\hat{w}(\hat{s}) < 0 < \hat{g}(\hat{s})$ for small $\hat{s} > 0$.
    Therefore, $\hat{w}(\hat{s}) < \hat{g}(\hat{s})$ for all $\hat{s} \in (0,\infty)$.
    The bound for $b_-$ now follows from the expansion \eqref{eqn:w-hat-s-expansion} of $\hat{w}$.
\end{proof}
\begin{remark}
Subsolutions such as those used in the above proof can be found algorithmically by the following procedure.
We compute the solution $w$ numerically given initial data and approximate it by a Taylor expansion on overlapping intervals.
One can tune the coefficients of the Taylor polynomial to ensure it satisfies the subsolution criterion and lies below, or agrees with, a given initial value. 
It is convenient to overshoot on the subsolution inequality and shrink the intervals in order to more easily verify the subsolution property, as exemplified above. 
The crucial subsolutions~\eqref{eqn:g-subsolution-for-b+} and \eqref{eqn:g-supersolution-for-b-}, valid in regions of the form $[s_*, \infty)$, then allow us to control the asymptotic behavior of the solution by approximating the equation near infinity.
The initial condition of $w$ is used to produce an intermediate subsolution on $[0,\ve]$, to move away from the singularity of the ODE.
The intermediate range $[\ve, s_*]$ can be broken up into intervals where Taylor approximations are valid.
This method works well with low degree polynomials only on short intervals; the barriers $g, \hat{g}$ in~\eqref{eqn:g-subsolution-for-b+} and \eqref{eqn:g-supersolution-for-b-} demonstrate this method.
\end{remark}
Using this refined asymptotic expansion, we can construct approximately minimal smoothings $\Sigma_{\delta}$ of the singular link $\Sigma_0$ of the cone $\mathbf{C}_0 \times \bR$.
Starting with the distinguished Jacobi field $\phi = z^3 r^{-2} - z$ on $\mathbf{C}_0 \times \bR$, we use an even cutoff function $\psi$ supported in a neighborhood of $\Sigma_0 \cap \{ z = 0 \}$ to define
\[
\zeta := \frac{\int_{\Sigma_0} \phi^2}{\int_{\Sigma_0}\psi(z)\phi^2}\psi(z)\phi.
\]
The surfaces $\Sigma_\delta$ will be constructed to have mean curvature equal to a constant multiple of $\zeta$, meaning that there exists a function $h(\delta)$, for small $|\delta| \leq \delta_0$, such that $H(\Sigma_\delta) = h(\delta) \zeta$.

We first construct auxiliary surfaces $\tilde{\Sigma}_{\delta}$ that desingularize $\Sigma_0$ by attaching appropriate rescaled leaves of $S_\pm$ at the two singular points to the graph of $\delta \phi$.
The link $\Sigma_0 \subset \bS^8$ is given by the set
\[
\{ (x,y,z) \in \bR^5 \times \bR^3 \times \bR : |x|^2 = 2 |y|^2, \; |x|^2 + |y|^2 + z^2 = 1 \} \subset \bS^8 \subset \bR^8 \times \bR
\]
and the two singular points arise at $p_{\pm} := \{ (x,y,z) = (0,0,\pm 1) \}$.
At $p_{\pm}$, we use the unit normal pointing into the region $\{ |x| > \sqrt{2} \, |y| \}$ and work in the chart on $\bS^8$ given by
\begin{equation}\label{eqn:Fpm}
F_{\pm} : \bR^8 \ni \xi \mapsto (\xi, \pm \sqrt{1 - |\xi|^2}).
\end{equation}
Then, $F^* \Sigma_0$ is the Lawson cone $\mathbf{C}_0 = \{ |x|^2 = 2 |y|^2 \}$ and the pullback of the spherical metric becomes $F^* g_{\bS^8} = g_{\on{Euc}} + O(r^2)$ in a neighborhood of the point, for small $r^2 = |\xi|^2$.
Near the singular points $p_{\pm}$, the decay of the Jacobi field $\phi$ matches the asymptotics of the surfaces $S_\pm$, since
\begin{equation}\label{eqn:match-Hardt-Simon}
\phi = z^3 r^{-2} - z = \pm r^{-2} (1-r^2)^{\frac{3}{2}} \mp (1-r^2)^{\frac{1}{2}} = \pm \, r^{-2} + O(1)
\end{equation}
and the leaves $S_{\pm}$ are the normal graphs of $\pm r (r^{-3} + b_{\pm} r^{-4} + O(r^{-6})) = \pm r^{-2} + O(r^{-3})$ over $\mathbf{C}_0$.
We let $r_{\delta} = \delta^{\frac{\alpha}{3}}$ for some $\alpha < 1$ sufficiently close to $1$ and define $r \in (r_{\delta}, 2 r_{\delta})$ to be the gluing region.
For $\lambda \in \bR$, we denote by $\lambda S := |\lambda| S_{\mr{sgn}(\lambda)}$, the appropriate rescalings of the Hardt-Simon leaves $S_+$ or $S_-$; notably, $\pm S = S_{\pm}$.
Using the charts \eqref{eqn:Fpm} to identify a neighborhood of each singular point with a ball in $\bR^8$, we define the smoothing $\tilde{\Sigma}_{\delta}$ to be the surface $\pm \delta^{\frac{1}{3}} S$ in the component of the region $r < r_{\delta}$ containing $p_\pm$ and to be the graph of $\delta \phi$ over $\Sigma_0$ in the region $r > 2 r_{\delta}$.
Since $\phi$ is an odd function, this construction requires choosing different leaves $S_{\pm}$ of the Hardt-Simon foliation on the two singular points $p_{\pm}$.

In the region $(r_{\delta}, 2 r_{\delta})$ we use a smooth transition function to interpolate between the two pieces.
We rescale the domain by $\tilde{r} := r_{\delta}^{-1} r$ to express $r_{\delta}^{-1} ( \pm \delta^{\frac{1}{3}} S)$ as the graph of
\[
\pm \left( r_{\delta}^{-3} \delta \, \tilde{r}^{-2} + b_{\pm} r_{\delta}^{-4} \delta^{\frac{4}{3}} \tilde{r}^{-3} \right) + O( r_{\delta}^{-1} \delta + r_{\delta}^{-9} \delta^3) 
\]
over $\mathbf{C}_0$, with respect to the spherical normal vector in the metric $r_{\delta}^{-2} F^* g_{\bS^8} = g_{\on{Euc}} + O(r_{\delta}^2)$ for $\tilde{r} \in (1,2)$.
The coefficient $b_{\pm}$ corresponds to the surface $\pm \delta^{\frac{1}{3}}S$, as defined above and glued in near the point $p_{\pm}$.
In view of the expansion \eqref{eqn:match-Hardt-Simon}, we can define $r_{\delta}^{-1}\tilde{\Sigma}_{\delta}$ on the region $\tilde{r} \in (1,2)$ to be the graph of a function $v_\delta$ given by
\begin{equation}\label{eqn:definition-gluing-region}
    v_{\delta}(\tilde{r}) := \begin{cases}
    r_{\delta}^{-3} \delta \, \tilde{r}^{-2} + \chi(\tilde{r}) b_{+} \, r_{\delta}^{-4} \delta^{\frac{4}{3}} + O( r_{\delta}^{-1} \delta + r_{\delta}^{-9} \delta^3) & \tilde{r} \in (1,2)\cap\{z > 0\}, \\
    -\left(r_{\delta}^{-3} \delta \, \tilde{r}^{-2} + \chi(\tilde{r}) b_{-} \, r_{\delta}^{-4} \delta^{\frac{4}{3}}\right) + O( r_{\delta}^{-1} \delta + r_{\delta}^{-9} \delta^3) & \tilde{r} \in (1,2)\cap\{z < 0\}, \\
    \end{cases}
\end{equation}
where $\chi$ is some smooth cutoff function with $\chi(s) = 1$ for $s \leq 1$ and $\chi(s) = 0$ for $s \geq 2$.

The linear analysis using appropriately defined weighted $C^{k,\alpha}$ space constructions and non-linear estimates of \cite{uniqueness-gabor}*{\S 3.2 -- 3.4} apply directly in our setting, where Propositions 3.5 -- 3.8 carry over to the space of $G$-invariant functions for $G = O(5) \times O(3)$.
Therefore, for all sufficiently small $|\delta| \leq \delta_0$ we obtain surfaces $\Sigma_{\delta}$ as graphs of functions $u \in C_{\tau}^{2,\alpha,G,0}$ over $\tilde{\Sigma}_{\delta}$ with small norm $\| u\|_{C^{2,\alpha}_{\tau}} \leq \delta^{1 + \frac{\kappa}{3}}$, whose mean curvature is a constant multiple of $\zeta$.
We furthermore normalize $u = 0$ along the ray $\{ r = r_0, z > 0\}$, for some small $r_0$ with $\phi(r_0) \neq 0$, to ensure the uniqueness of the surfaces $\Sigma_{\delta}$.
The family $\Sigma_\delta$ is the replacement for the nearby minimal surfaces generated by the integrable Jacobi fields on the cones studied by Simon \cite{uniqueness-cylindrical}.
\begin{proposition}\label{prop:approximate-Sigma-delta}
For all $|\delta| \leq \delta_0$ sufficiently small, the mean curvature of the smoothings $\Sigma_{\delta}$ of $\Sigma_0$ is given by $H(\Sigma_{\delta}) = h(\delta) \zeta$, for a function $h(\delta)$ satisfying
    \begin{equation}\label{eqn:h-properties}
h(\delta) = c_1 (b_+ + b_-) \delta^{\frac{4}{3}} + O\big(|\delta|^{\frac{4}{3}+\kappa}\big) \quad \text{and}\quad h'(\delta) = \tfrac{4}{3} c_1 (b_+ + b_-) \delta^{\frac{1}{3}} + O\big(|\delta|^{\frac{1}{3}+\kappa}\big)
\end{equation}
for some small $\kappa > 0$, some constant $c_1 > 0$ depending on the geometry of the link $\Sigma_0$.
\end{proposition}
\begin{proof}
Following the strategy of \cite{uniqueness-gabor}*{\S 3.5}, we first note that
\begin{equation}\label{eqn:integral-over-sigma-delta-1}
\int_{\tilde{\Sigma}_{\delta}} H(\tilde{\Sigma}_{\delta}) \phi \, dA = c_0 (b_+ + b_-) \delta^{\frac{4}{3}} + O( \delta^{\frac{4}{3}+ \kappa'})
\end{equation}
for some $\kappa' > 0$, where $c_0$ denotes the volume of the link of $\mathbf{C}_0$.
The computation of \cite{uniqueness-gabor}*{Proposition 3.9}, applies to our setting, showing that the main contribution to the above integral comes from the gluing regions $r \in (r_{\delta}, 2 r_{\delta})$ near each singular point $p_{\pm}$.
Working in the rescaled variable $\tilde{r} := r_{\delta}^{-1} r \in (1,2)$, we combine the expressions \eqref{eqn:match-Hardt-Simon} for $\phi$ and \eqref{eqn:definition-gluing-region} for $v_{\delta}$ to obtain
\begin{align*} H(\tilde{\Sigma}_{\delta}) \phi &= b_{\pm} r_{\delta}^{-4} \delta^{\frac{4}{3}} \, \cL_{g_0} ( \chi(\tilde{r}) \tilde{r}^{-3}) \cdot \tilde{r}^{-2} + O( r_{\delta}^{-1} \delta + r_{\delta}^{-6} \delta^2) \\
&= b_{\pm} r_{\delta}^{-4} \delta^{\frac{4}{3}} \, \chi''(\tilde{r}) \, \tilde{r}^{-3} + O(r_{\delta}^{-1} \delta + r_{\delta}^{-6} \delta^2),
\end{align*}
for $\cL_{g_0}$, the Jacobi operator over $\mathbf{C}_0$ with respect to the metric $g_0 = g_{\on{Euc}}|_{\mathbf{C}_0}$.
Choosing $r_{\delta} = \delta^{\frac{\alpha}{3}}$ for $\alpha < 1$ sufficiently close to $1$ guarantees that the error terms and the contributions from the other regions are $O(\delta^{\frac{4}{3} + \kappa'})$, for some $\kappa' > 0$, while the leading contribution to the integral of \eqref{eqn:integral-over-sigma-delta-1} is $c_0 b_{\pm} \delta^{\frac{4}{3}}$ in the gluing region near $p_{\pm}$.
This proves equality \eqref{eqn:integral-over-sigma-delta-1}.
The arguments of \cite{uniqueness-gabor}*{Proposition 3.10, Corollary 3.11} now apply to our setting, showing that $h(\delta) = c_1 (b_+ + b_-) \delta^{\frac{4}{3}} + O(\delta^{\frac{4}{3} + \kappa})$ for some $c_1, \kappa > 0$.
Since $(b_+ + b_-) < 0$ due to Proposition \ref{prop:main-expansion}, the equality \eqref{eqn:h-properties} follows.
\end{proof}

\vspace{-0.05in}

Having obtained the key properties of the approximate smoothings $\Sigma_{\delta}$ in Proposition \ref{prop:approximate-Sigma-delta}, the proof of the uniqueness of $\mathbf{C}_0 \times \bR$ follows the same strategy as for $C_{3,3} \times \bR$.
This strategy can prove the uniqueness of the tangent cone $C \times \bR$ for any strictly stable, strictly area-minimizing cone $C$, including $C (\bS^2 \times  \bS^4)$, that admits the refined asymptotics of Proposition~\ref{prop:main-expansion}, mainly the estimate $(b_+ + b_-) \neq 0$.
We review the main steps below and refer the reader to \cite{uniqueness-gabor}*{\S 4 -- 6} for details.

\begin{proof}[Proof of Theorem~\ref{thm}]
The expansion \eqref{eqn:h-properties} allows us to apply the arguments of \cite{uniqueness-gabor}*{\S 3.6} to the cone $\mathbf{C}_0$.
The surfaces $\Sigma_{\delta}$ have strictly larger area than $\Sigma_0$ for sufficiently small $\delta >0$, and satisfy
\[
\on{Area}(\Sigma_{\delta}) \leq \on{Area}(\Sigma_0) + C |\delta|  |h(\delta)|.
\]
There exist functions $\phi_{\delta}, \xi_{\delta}$ on $\Sigma_{\delta}$, for $|\delta| \leq \delta_0$, satisfying
\begin{equation}\label{eqn:phi-delta-xi-delta}
    \cL_{\Sigma_{\delta}} \phi_{\delta} = h'(\delta) \zeta, \qquad \cL_{\Sigma_{\delta}} \xi_{\delta} = \zeta - c_{\delta} \phi_{\delta}, \qquad \la \xi_{\delta}, \phi_{\delta} \rg_{L^2(\Sigma_{\delta})} = 0,
\end{equation}
and $\phi_{\delta} = \phi$ along $\{ r = r_0, z > 0 \}$, for the $r_0$ used in constructing $\Sigma_{\delta}$.
These extend to degree-one homogeneous functions on $V_{\delta}$, and we define the cone $W_{\delta}$ as the graph of $- h(\delta) \xi_{\delta}$ over $V_{\delta}$.

The cone $V_{\delta} = C(\Sigma_{\delta}) \subset \bR^8 \times \bR$ has mean curvature $H(V_{\delta}) = h(\delta) \zeta \rho^{-2}$, where $\rho := (|x|^2 + |y|^2 + z^2)^{\frac{1}{2}} = (r^2 + z^2)^{\frac{1}{2}}$ is the distance from the origin in $\bR^8 \times \bR$.
Since $H(V_{\delta}) \neq 0$ for small $\delta \neq 0$, by~\eqref{eqn:h-properties}, we can construct comparison surfaces $T_{\delta}$ as small minimal perturbations of $V_{\delta}$ in a suitable weighted $C^{2,\alpha}_{\gamma, \tau}$ sense (cf.~\cite{uniqueness-gabor}*{Definition 4.1}) over annuli of the form $|\log \rho| < |\delta|^{- \kappa}$ for small $\kappa > 0$.
Our setting uses the corresponding $G$-invariant function spaces $C^{k,\alpha,G}_{\gamma, \tau}$, where $G := O(5) \times O(3)$ and the indicial interval $\tau \in (-3, -2)$ for $C_{2,4}$ is the same as for $C_{3,3}$. 
The arguments of Propositions 4.3 -- 4.5 of \cite{uniqueness-gabor} apply in this context: there is an $\ve > 0$ such that for all $|\delta| \leq \delta_0$, there exists a function $f$ on $\{ |\log \rho| < |\delta|^{- \kappa} \} \cap V_{\delta}$ satisfying $\| f \|_{C^{2,\alpha}_{1,\tau}} \leq |h(\delta)| \cdot |\delta|^{\ve}$, and the normal graph of $- h(\delta) \left( 7^{-1} \phi_{\delta} \log \rho + \xi_{\delta} \right) + f$ over $V_{\delta}$ is a minimal surface $T_{\delta}$.
For $|\log \Lambda| < |\delta|^{-\kappa}$, the surface $\Lambda T_{\delta}$ is the graph of a function over $W_{\delta}$ with bounded $C^{2,\alpha}_{1,-2}(B_1 \setminus B_{1/2})$ norm.
These results imply a \L{}ojasiewicz inequality as in \cite{uniqueness-gabor}*{Proposition 4.7}, producing a uniform $C>0$~with
\begin{equation}\label{eqn:lojasiewicz-monotonicity}
\int_{T_{\delta} \cap (B_{1/2} \setminus B_{1/4})} \frac{|X^{\perp}|^2}{|X|^{10}} > C^{-1} |h(\delta)|^2
\end{equation}
For $|\delta| \leq \delta_0$, we denote by $\cT, \cW$ the sets of all comparison surfaces $T_{\delta}, W_{\delta}$ and their rotations in $\bR^9$.
Let $\cM$ denote the class of least-area oriented boundaries in the ball $B_2(0) \subset \bR^9$. 
For $T =\Lambda T_{\delta}\in \cT$ with $|\log \Lambda| \le |\delta|^{-\kappa}$, we may define a quantity $D_T(M)$ that measures the distance of an $M \in \cM$ relative to $T$ in a fixed annulus $B_1\setminus B_{\rho_0}$ as in \cite{uniqueness-gabor}*{5.1 -- 5.5}.
Using a barrier argument, we can deduce the necessary non-concentration estimate for the distance $D_T(M)$, analogous to the $L^2$ non-concentration in \cite{uniqueness-cylindrical}.
We thus arrive at a ``geometric three-annulus Lemma" as in~\cite{uniqueness-gabor}*{Proposition 5.12}, there exist some $\alpha, L>0$ such that if $D_T(M)$ is sufficiently small, then
    \begin{enumerate}[(i)]
       \item if $D_{LT}(LM)\ge L^{\alpha}D_T(M)$, then $D_{L^2T}(L^2M)\ge L^{\alpha} D_{LT}(LM)$, and
        \item if $D_{L^{-1}T}(L^{-1}M)\ge L^{\alpha}D_T(M)$, then $D_{L^{-2}T}(L^{-2}M)\ge L^{\alpha} D_{L^{-1}T}(L^{-1}M)$.
    \end{enumerate}
Supposing otherwise, a blowup argument would produce a non-trivial Jacobi field on $\mathbf{C}_0 \times \bR$ violating Simon's $L^2$ three-annulus Lemma \cite{SimonAnnals81}. 
For the constant $L>0$ and $\beta>0$, we define the quantity
\begin{equation}\label{eqn:E-beta}
    E_{\beta}(M) := \inf \{ D_W(M) + D_W(L^{\beta} M) : W \in \cW \},
\end{equation}
which controls the $\cF$-distance (in the sense of currents) of $M$ and $L^{\beta} M$ using the cones $\cW$, since $d_{\cF} (M, L^{\beta} M) \leq C \cdot E_{\beta}(M)$ for any minimal surface $M \in \cM$ in $B_2$ with $\on{Area}(M) < 2 \cdot \on{Area}(B_2 \cap (\mathbf{C}_0 \times \bR))$.
Using \eqref{eqn:lojasiewicz-monotonicity}, we obtain a \L{}ojasiewicz inequality for the area excess
\[
\cA(M) := \on{Area}(M \cap B_{1/2}) - \on{Area}( (\mathbf{C}_0 \times \bR) \cap B_{1/2})
\]
of any $M \in \cM$.
This has the form
\begin{equation}\label{eqn:lojasiewicz-simon}
    \cA(M) - \cA(2M) > C^{-1} |h(\delta)|^2 \quad \text{and} \quad \cA(M)^{\theta} - \cA(2M)^{\theta} > C^{-1} |h(\delta)|
\end{equation}
for sufficiently small $\theta > 0$, as in \cite{uniqueness-gabor}*{Proposition 6.3}.
Combining the geometric three-annulus lemma with the \L{}ojasiewicz inequality \eqref{eqn:lojasiewicz-simon} then proves that any minimal surface $M \in \cM$ in $B_1$ sufficiently Hausdorff-close to $\mathbf{C}_0 \times \bR$ and with the same density at the origin satisfies
\[
E_{\beta}(L^{\beta} M) \leq \tfrac{1}{2} E_{\beta} (M) \quad \text{or} \quad E_{\beta}(L^{\beta} M) \leq C \left( \cA(L^{\beta} M)^{\theta} - \cA(2 \, L^{\beta} M)^{\theta} \right)
\]
as in \cite{uniqueness-gabor}*{Propositions 6.5 -- 6.6}.
Iterating this procedure produces an inequality of the form
\[
d_{\cF} (M, L^{(N+1)\beta} M) \leq C ( E_{\beta}(M) + \cA(M)^{\theta})
\]
and since $\mathbf{C}_0 \times \bR$ is a tangent cone, we have $d_{\cF} (L^{N_i \beta} M, \mathbf{C}_0 \times \bR) \to 0$ for a sequence $N_i \to \infty$, thus $d_{\cF}(L^{k \beta} M, \mathbf{C}_0 \times \bR) \to 0$ as $k \to \infty$.
Therefore, $\mathbf{C}_0 \times \bR$ is the unique tangent cone at the origin.
\end{proof}

\nocite{*}
\bibliography{ref}

\end{document}